\documentclass{amsart}
\usepackage{amsfonts,amssymb,amsmath,amsthm,mathrsfs}
\usepackage{url}
\usepackage{enumerate}

\newtheorem{theorem}{Theorem}[section]
\newtheorem{lemma}[theorem]{Lemma}

\theoremstyle{definition}

\newtheorem{remark}[theorem]{Remark}
\numberwithin{equation}{section}
\author[J. Chen]{Jiecheng Chen}
\address{Jiecheng Chen, Department of Mathematics, Zhejiang Normal University, Jinhua,
321004, P. R. China}
\email{jcchen@zjnu.edu.cn}
\author[G. Hu]{Guoen Hu}
\address{Guoen Hu, Department of Mathematics, School of Science, Zhejiang University of Science and Technology,
	Hangzhou 310023, People's Republic of China}
\email{guoenxx@163.com}
\author[X. Tao]{Xiangxing Tao}
\address{Xiangxing Tao, Department of Mathematics, School of Science, Zhejiang University of Science and Technology,
	Hangzhou 310023, People's Republic of China}
\email{xxtao@zust.edu.cn}

\thanks{The research of the first author was supported by the NNSF of
China under grant \#12071437, the research of the second  (corresponding) author was supported by
the NNSF of
China under grants \#11871108, and the third author was supported by the NNSF of
China under grant \#11771339.}

\begin{document}

\title[Calder\'on commutator]{$L^p(\mathbb{R}^d)$ boundedness for the Calder\'on commutator with rough kernel}

\begin{abstract}
Let $k\in\mathbb{N}$, $\Omega$ be homogeneous of degree zero, integrable on $S^{d-1}$ and have vanishing moment of order $k$, $a$ be a function on $\mathbb{R}^d$ such that $\nabla a\in L^{\infty}(\mathbb{R}^d)$, and $T_{\Omega,\,a;k}$ be the  $d$-dimensional Calder\'on commutator defined  by
$$T_{\Omega,\,a;k}f(x)={\rm p.\,v.}\int_{\mathbb{R}^d}\frac{\Omega(x-y)}{|x-y|^{d+k}}\big(a(x)-a(y)\big)^kf(y){d}y.$$
In this paper, the authors prove that if $$\sup_{\zeta\in S^{d-1}}\int_{S^{d-1}}|\Omega(\theta)|\log ^{\beta} \big(\frac{1}{|\theta\cdot\zeta|}\big)d\theta<\infty,$$ with $\beta\in(1,\,\infty)$,
then for $\frac{2\beta}{2\beta-1}<p<2\beta$, $T_{\Omega,\,a;\,k}$ is bounded on $L^p(\mathbb{R}^d)$. \end{abstract}
\maketitle
\section{Introduction}
We will work on $\mathbb{R}^d$, $d\geq 2$.
Let $k\in\mathbb{N}$, $\Omega$ be homogeneous of degree zero, integrable on  $S^{d-1}$, the unit sphere in $\mathbb{R}^d$,  and have vanishing moment of order $k$, that is,  for all multi-indices $\gamma\in\mathbb{Z}_+^d$,
\begin{eqnarray}\label{equa:1.1}\int_{S^{d-1}}\Omega(\theta)\theta^{\gamma}d\theta=0,\,\,\,|\gamma|=k.\end{eqnarray} Let $a$ be a function on $\mathbb{R}^d$ such that $\nabla a\in L^{\infty}(\mathbb{R}^d)$. Define the  $d$-dimensional Calder\'on commutator $T_{\Omega, a;\,k}$ by
\begin{eqnarray}\label{eq:1.2} T_{\Omega, a;k}f(x)={\rm p. v.}\int_{\mathbb{R}^d}\frac{\Omega(x-y)}{|x-y|^{d+k}}\big(a(x)-a(y)\big)^kf(y)dy.\end{eqnarray}
For simplicity, we denote $T_{\Omega,a;\,1}$ by $T_{\Omega,a}$. Commutators of this type were introduced by Calder\'on \cite{cal1}, who proved that if $\Omega\in L\log L(S^{d-1})$, then $T_{\Omega,a}$ is bounded on $L^p(\mathbb{R}^d)$ for all $p\in (1,\,\infty)$. It should be pointed out that Calder\'on's result in \cite{cal1} also holds for $T_{\Omega,a;\,k}$.
Pan, Wu and Yang \cite{pwy} improved Calder\'on's result, and obtained the following conclusion.
\begin{theorem}\label{thm1.-1}
Let $\Omega$ be homogeneous of degree zero, satisfy the vanishing moment (\ref{equa:1.1}) with $k=1$, $a$ be a   function on $\mathbb{R}^d$ such that $\nabla a\in L^{\infty}(\mathbb{R}^d)$. Suppose that $\Omega\in H^1(S^{d-1})$ (the Hardy space on $S^{d-1}$), then
$T_{\Omega,a}$ is bounded on $L^p(\mathbb{R}^d)$ for all $p\in (1,\,\infty)$.
\end{theorem}

Chen, Ding and Hong \cite{cdh} showed that the converse of Theorem \ref{thm1.-1} is also true. Precisely, Chen at al. \cite[p.\,1501]{cdh} established the following result.
\begin{theorem}
Let $\Omega$ be homogeneous of degree zero, $\Omega\in {\rm Lip}_{\alpha}(S^{d-1})$ for some $\alpha\in (0,\,1]$, and satisfy the vanishing moment (\ref{equa:1.1}) with $k=1$, $a\in L_{{\rm loc}}^1(\mathbb{R}^d)$. If $T_{\Omega,\,a}$ is bounded on $L^p(\mathbb{R}^d)$ for some $p\in (1,\,\infty)$, then $\nabla a\in L^{\infty}(\mathbb{R}^d)$.
\end{theorem}

Hofmann \cite{hof1}
considered the weighted $L^p$ boundedness with $A_p$ weights for $T_{\Omega,a;\,k}$, and proved that if $\Omega\in L^{\infty}(S^{d-1})$ and satisfies (\ref{equa:1.1}), then for $p\in (1,\,\infty)$ and $w\in A_p(\mathbb{R}^d)$,  $T_{\Omega,a;k}$ is bounded on $L^p(\mathbb{R}^d,\,w)$, where and in the following, $A_p(\mathbb{R}^d)$ denotes the weight function class of Muckenhoupt, see \cite[Chap. 9]{gra} for the definition and properties of $A_p(\mathbb{R}^d)$. Ding and Lai \cite{dinglai} considered the weak type endpoint estimate for $T_{\Omega, a}$, and proved that $\Omega\in L\log L(S^{d-1})$ is a sufficient condition such that $T_{\Omega,a}$ is bounded from $L^1(\mathbb{R}^d)$ to $L^{1,\,\infty}(\mathbb{R}^d)$.

For $\beta\in [1,\,\infty)$, we say that $\Omega\in GS_{\beta}(S^{d-1})$ if $\Omega\in L^1(S^{d-1})$ and
\begin{eqnarray}\label{equation1.3}
\sup_{\zeta\in S^{d-1}}\int_{S^{d-1}}|\Omega(\theta)|\log ^{\beta}\big(\frac{1}{|\zeta\cdot \theta|}\big)d\theta<\infty.
\end{eqnarray}
The condition (\ref{equation1.3}) was introduced by Grafakos
and Stefanov \cite{gras} in order to study the $L^p(\mathbb{R}^d)$ boundedness for the homogeneous singular integral operator defined by
\begin{eqnarray}\label{equation1.4}T_{\Omega}f(x)={\rm p.\, v.}\int_{\mathbb{R}^d}\frac{\Omega(x-y)}{|x-y|^d}f(y)dy,
\end{eqnarray}
where $\Omega$ is homogeneous of degree zero and has mean value zero  on $S^{d-1}$. Obviously,
$L(\log L)^{\beta}(S^{d-1})\subset GS_{\beta}(S^{d-1})$. On the other hand, as it was pointed out in \cite{gras},
there exist integrable functions
on $S^{d-1}$
which are not in $H^1(S^{d-1})$   but satisfy (\ref{equation1.3}) for all $\beta\in (1,\,\infty)$. Thus, it is of interest to consider the $L^p(\mathbb{R}^d)$ boundedness for operators such as $T_{\Omega}$ and $T_{\Omega,a;\,k}$ when $\Omega\in GS_{\beta}(S^{d-1})$.
Grafakos
and Stefanov \cite{gras} proved that if $\Omega\in GS_{\beta}(S^{d-1})$ for some $\beta\in (1,\,\infty]$, then $T_{\Omega}$ is bounded on $L^p(\mathbb{R}^d)$ for $1+1/\beta<p<1+\beta$. Fan, Guo and Pan \cite{fgp} improved the result of \cite{gras}, and proved the following result.

\begin{theorem}\label{thm1.0}
Let $\Omega$ be homogeneous of degree zero, integrable and have mean value zero  on $S^{d-1}$. Suppose that $\Omega\in GS_{\beta}(S^{d-1})$ with $\beta\in (1,\,\infty)$, then for $\frac{2\beta}{2\beta-1}<p<2\beta$, $T_{\Omega}$ is bounded on $L^p(\mathbb{R}^d)$.
\end{theorem}

The purpose of this paper is to establish  the $L^p(\mathbb{R}^d)$ boundedness of $T_{\Omega,\,a;k}$ when $\Omega\in GS_{\beta}(S^{d-1})$ for some $\beta>1$. Our main result can be stated as follows.

\begin{theorem}\label{thm1.1} Let $k\in\mathbb{N}$, $\Omega$ be homogeneous of degree zero, satisfy the vanishing moment (\ref{equa:1.1}), $a$ be a   function on $\mathbb{R}^d$ such that $\nabla a\in L^{\infty}(\mathbb{R}^d)$. Suppose that $\Omega\in GS_{\beta}(S^{d-1})$ with $\beta\in (1,\,\infty)$,
Then for $\frac{2\beta}{2\beta-1}<p<2\beta$, $T_{\Omega,\,a;\,k}$ is bounded on $L^p(\mathbb{R}^d)$.
\end{theorem}

Different from the operator $T_{\Omega}$ defined by (\ref{equation1.4}), $T_{\Omega,a;\,k}$ is not a convolution operator, and the argument in \cite{gras,fgp} does not apply to $T_{\Omega,\,a;\,k}$ directly. To prove Theorem \ref{thm1.1}, we will first prove the $L^2(\mathbb{R}^d)$ boundedness of $T_{\Omega,a;k}$ by employing the ideas used in \cite{hof1}, together with some new localizations and decompositions.  The  argument in the proof of $L^2(\mathbb{R}^d)$ boundedness is based on a refined decomposition appeared in (\ref{equation2.decomp}). To prove the $L^p(\mathbb{R}^d)$ boundedness of $T_{\Omega,a;\,k}$, we will introduce a suitable approximation to $T_{\Omega,a;\,k}$ by a sequence of integral operators, whose kernels enjoy H\"ormander's condition.
We remark that the idea approximating  rough convolution operators by smooth operators was originated by Watson \cite{wat}.

In what follows, $C$ always denotes a
positive constant that is independent of the main parameters
involved but whose value may differ from line to line. We use the
symbol $A\lesssim B$ to denote that there exists a positive constant
$C$ such that $A\le CB$.  Constant with subscript such as $C_1$,
does not change in different occurrences. For any set $E\subset\mathbb{R}^d$,
$\chi_E$ denotes its characteristic function.  For a cube
$Q\subset\mathbb{R}^d$ and $\lambda\in(0,\,\infty)$,
$\lambda Q$  denotes the cube with the same center as $Q$  whose
side length is $\lambda$ times that of $Q$. For a suitable function $f$, we denote $\widehat{f}$ the Fourier transform of $f$. For $p\in [1,\,\infty]$, $p'$ denotes the dual exponent of $p$, namely, $p'=p/(p-1)$.

\section{Proof of Theorem \ref{thm1.1}: $L^2(\mathbb{R}^d)$ boundedness}
This section is devoted to the proof of the $L^2(\mathbb{R}^d)$ boundedness of $T_{\Omega,a;\,k}$. For simplicity, we only consider the case $k=1$. As it was pointed out in \cite[Section 2]{hof1}, the argument in this section still works for all $k\in\mathbb{N}$, if we proceed by induction on the order $k$.

Let $\phi\in C^{\infty}_0(\mathbb{R}^d)$ be a radial function, ${\rm supp}\,\phi\subset B(0,\,2)$, $\phi(x)=1$ when $|x|\leq 1$. Set $\varphi(x)=\phi(x)-\phi(2x)$. We then have that
\begin{eqnarray}\label{eq2.varphi}
\sum_{j\in\mathbb{Z}}\varphi(2^{-j}x)\equiv 1,\,\,|x|>0.\end{eqnarray}
Let $\varphi_j(x)=\varphi(2^{-j}x)$ for $j\in\mathbb{Z}$.

For a function $\Omega\in L^1(S^{d-1})$, define the operator $W_{\Omega, j}$ by
\begin{eqnarray}\label{equation2.womega}
W_{\Omega,j}h(x)=\int_{\mathbb{R}^d}\frac{\Omega(x-y)}{|x-y|^{d+1}}\varphi_j(x-y)h(y)dy.
\end{eqnarray}

\begin{lemma}\label{lem2.2}Let $\Omega$ be homogeneous of degree zero, integrable on $S^{d-1}$, satisfy the vanishing moment (\ref{equa:1.1}) with $k=1$ and $\Omega\in GS_{\beta}(S^{d-1})$ for some $\beta\in (1,\,\infty)$, $a$ be a function  on $\mathbb{R}^d$ such that $\nabla a\in L^{\infty}(\mathbb{R}^d)$. Then for any $r\in (0,\,\infty)$, functions ${\eta}_1,\,{\eta}_2\in C^{\infty}_0(\mathbb{R}^d)$ which are supported on balls of radius no larger than $r$,
$$\Big|\int_{\mathbb{R}^d}\eta_2(x)T_{\Omega,\,a}\eta_1(x)dx\Big|\lesssim \|\Omega\|_{L^1(S^{d-1})}
			r^{-d}\prod_{j=1}^2\big(\|\eta_j\|_{L^{\infty}(\mathbb{R}^d)}+r\|\nabla \eta_j\|_{L^{\infty}(\mathbb{R}^d)}\big).
$$
\end{lemma}
Recall that under the hypothesis of Lemma \ref{lem2.2}, the operator $T_{\Omega,m}$ defined by
\begin{eqnarray}\label{equa2.tam}T_{\Omega,m}f(x)={\rm p.\,v.}\int_{\mathbb{R}^d}\frac{\Omega(x-y)(x_m-y_m)}{|x-y|^{d+1}}f(y)dy,\,\,1\leq m\leq d
\end{eqnarray}
is bounded on $L^2(\mathbb{R}^d)$ (see \cite{gras}). Lemma \ref{lem2.2} can be proved by repeating the proof of Lemma 2.5 in \cite{hof1}.

Let $\psi\in C^{\infty}_0(\mathbb{R}^d)$ be a radial function, have integral zero and ${\rm supp}\,\psi\subset B(0,\,1)$. Let $Q_s$ be the operator defined by  $Q_sf(x)=\psi_s*f(x)$, where $\psi_s(x)=s^{-d}\psi(s^{-1}x)$. We assume that $$\int^{\infty}_0[\widehat{\psi}(s)]^4\frac{ds}{s}=1.$$
Then the Calder\'on reproducing formula
\begin{eqnarray}\label{equation2.reproducing}\int^{\infty}_0Q_s^4\frac{ds}{s}=I\end{eqnarray}
holds true. Also, the  Littlewood-Paley theory tells us that
\begin{eqnarray}\label{equa2.little}\Big\|\Big(\int_{0}^{\infty}|Q_sf|^2\frac{ds}{s}\Big)^{1/2}\Big\|_{L^2(\mathbb{R}^d)}\lesssim \|f\|_{L^2(\mathbb{R}^d)}.
\end{eqnarray}

For each fixed $j\in\mathbb{Z}$, set
$$
T_{\Omega,\,a}^{j}f(x)=\int_{\mathbb{R}^d}K_{j}(x,\,y)f(y)dy,
$$
where
$$K_{j}(x,\,y)=\frac{\Omega(x-y)}{|x-y|^{d+1}}(a(x)-a(y)\big)\varphi_j(|x-y|).
$$
\begin{lemma}\label{lem2.3}
Let $\Omega$ be homogeneous of degree zero, integrable on $S^{d-1}$ and $\Omega\in GS_{\beta}(S^{d-1})$ for some $\beta\in (1,\,\infty)$, then for $j\in\mathbb{Z}$ and $0<s\leq 2^j$,
$$\|Q_sW_{\Omega,\,j}f\|_{L^2(\mathbb{R}^d)}\lesssim 2^{-j}\log^{-\beta}(2^j/s+1)\|f\|_{L^2(\mathbb{R}^d)}.
$$
\end{lemma}
\begin{proof}Let $K_{\Omega,j}(x)=\frac{\Omega(x)}{|x|^{d+1}}\varphi_j(|x|)$. By Plancherel's theorem, it suffices to prove that
\begin{eqnarray}\label{equation2.7}|\widehat{\psi_s}(\xi)\widehat{K_{\Omega,j}}(\xi)|\lesssim 2^{-j}\log^{-\beta}(2^j/s+1).
\end{eqnarray}
As it was proved by Grafakos and Stefanov \cite[p. 458]{gras}, we know that
$$|\widehat{K_{\Omega,j}}(\xi)|\lesssim 2^{-j}\log^{-\beta}(|2^j\xi|+1).$$
On the other hand, it is easy to verify that
$$|\widehat{\psi_s}(\xi)|\lesssim \min\{1,\,|s\xi|\}.$$
Observe that (\ref{equation2.7}) holds true when $|2^j\xi|\leq 1$, since
$$|s\xi|\log^{-\beta}(2^j|\xi|+1)=\frac{s}{2^j}|2^j\xi|\log^{-\beta}(|2^j\xi|+1)\lesssim \frac{s}{2^j}\lesssim \log^{-\beta}(2^j/s+1).
 $$
 If $|s\xi|\geq 1$, we certainly have that
$$|\widehat{\psi_s}(\xi)\widehat{K_{\Omega,j}}(\xi)|\lesssim 2^{-j}\log^{-\beta}(2^j|\xi|+1)\lesssim 2^{-j}\log^{-\beta}(2^j/s+1).
$$
Now we assume that $s|\xi|<1$ and $|2^j\xi|>1$, and
$$2^{-k}2^j<s\leq 2^{-k+1}2^j,\,\,\,2^{k_1-1}<|\xi|\leq 2^{k_1}$$
for $k\in\mathbb{N}$ and $k_1\in \mathbb{Z}$ respectively. Then $j+k_1\in\mathbb{N}$, $j+k_1\leq k$ and
$$|s\xi|\log^{-\beta}(2^j|\xi|+1)\lesssim 2^{j-k+k_1}(j+k_1)^{-\beta}\lesssim k^{-\beta}\lesssim  \log^{-\beta}(2^j/s+1).
 $$
This verifies (\ref{equation2.7}).
\end{proof}

\begin{lemma}\label{lem2.4}Let $\Omega$ be homogeneous of degree zero, satisfy the vanishing moment (\ref{equa:1.1}) with $k=1$ and $\Omega\in GS_{\beta}(S^{d-1})$ for some $\beta\in (1,\,\infty)$, $a$ be a function on $\mathbb{R}^d$ with $\nabla a\in L^{\infty}(\mathbb{R}^d)$. Then
\begin{itemize}
\item[\rm (i)]$T_{\Omega, a}1\in{\rm BMO}(\mathbb{R}^d)$;
\item[\rm (ii)] for any $j\in\mathbb{Z}$ and $s\in (0,\,2^j]$;
$$\|Q_sT_{\Omega,a}^{j}1\|_{L^{\infty}(\mathbb{R}^d)}\lesssim \|\Omega\|_{L^1(S^{d-1})}2^{-j}s.$$
\end{itemize}
\end{lemma}
Conclusion (ii) is just Lemma 2.4 in \cite{hof1}, while (i) of Lemma \ref{lem2.4} can be proved by mimicking the proof of Lemma 2.3 in \cite{hof1}, since for all $1\leq m\leq d$, $T_{\Omega, m}$ defined by (\ref{equa2.tam}) is bounded on $L^2(\mathbb{R}^d)$ when $\Omega\in GS_{\beta}(S^{d-1})$ for $\beta>1$. We omit the details for brevity.

\medskip

{\it Proof of Theorem \ref{thm1.1}: $L^2(\mathbb{R}^d)$ boundedness}.
By (\ref{equation2.reproducing}), it suffices to prove that for $f,\,g\in C^{\infty}_0(\mathbb{R}^d)$,
\begin{eqnarray}\label{eq2.x1} \Big|\int^{\infty}_0\int_{0}^t\int_{\mathbb{R}^d}Q_s^4{T}_{\Omega,\,a}Q_t^4f(x)g(x)dx
\frac{ds}{s}\frac{dt}{t}\Big|\lesssim \|f\|_{L^2(\mathbb{R}^d)}\|g\|_{L^{2}(\mathbb{R}^d)},
\end{eqnarray}
and
\begin{eqnarray}\label{eq2.x2} \Big|\int^{\infty}_0\int_{t}^{\infty}\int_{\mathbb{R}^d}Q_s^4{T}_{\Omega,\,a}Q_t^4f(x)g(x)dx
\frac{ds}{s}\frac{dt}{t}\Big|\lesssim \|f\|_{L^2(\mathbb{R}^d)}\|g\|_{L^{2}(\mathbb{R}^d)}.
\end{eqnarray}
Observe that (\ref{eq2.x2}) can be deduced from (\ref{eq2.x1}) and a standard duality argument. Thus, we only need to prove (\ref{eq2.x1}).

We now prove (\ref{eq2.x1}). Without loss of generality, we assume that $\|\nabla a\|_{L^{\infty}(\mathbb{R}^d)}=1$. Write
\begin{eqnarray*}
&&\int^{\infty}_0\int_{0}^t\int_{\mathbb{R}^d}Q_s^4{T}_{\Omega,\,a}
Q_t^4f(x)g(x)dx\frac{ds}{s}\frac{dt}{t}\\
&&\quad=\sum_{j\in\mathbb{Z}}\int^{2^j}_0\int_{0}^t\int_{\mathbb{R}^d}Q_s{T}_{\Omega,\,a}^{j}Q_t^4f(x)Q_s^3g(x)dx
\frac{ds}{s}\frac{dt}{t}\\
&&\qquad+\sum_{j\in\mathbb{Z}}\int^{\infty}_{2^j}\int_{0}^{(2^jt^{\alpha-1})^{\frac{1}{\alpha}}}
\int_{\mathbb{R}^d}Q_s{T}_{\Omega,\,a}^{j}Q_t^4f(x)Q_s^3g(x)dx\frac{ds}{s}\frac{dt}{t}\\
&&\qquad+\sum_{j\in\mathbb{Z}}\int^{\infty}_{2^j}\int_{(2^jt^{\alpha-1})^{\frac{1}{\alpha}}}^t
\int_{\mathbb{R}^d}Q^4_s{T}_{\Omega,\,a}^{j}Q_t^4f(x)g(x)dx\frac{ds}{s}\frac{dt}{t}:={\rm D}_1+{\rm D}_2+{\rm D}_3,
\end{eqnarray*}		
where $\alpha\in (\frac{d+1}{d+2},\,1)$ is a constant.

We first consider term ${\rm D}_2$. For each fixed $j\in\mathbb{Z}$, let $\{I_{j,l}\}_{l}$ be a sequence of cubes having disjoint interiors and side length $2^j$, such that $
\mathbb{R}^d=\cup_{l}I_{j,l}.$
For each fixed $j,l$, let $\omega_{j,l}\in C^{\infty}_0(\mathbb{R}^d)$ such that ${\rm supp}\,\omega_{j,l}\subset 48dI_{j,l}$, $0\leq \omega_{j,l}\leq 1$ and $\omega_{j,l}(x)\equiv 1$ when $x\in 32dI_{j,l}$. Let $I_{j,l}^*=64dI_{j,l}$ and $x_{j,l}$ be the center of $I_{j,l}$. For each $l$, set $a_{j,l}(y)=(a(y)-a(x_{j,l}))\omega_{j,l}(y)$, and $h_{s,j,l}(y)=Q_s^2g(y)\chi_{I_{j,l}}(y)$.
It is obvious that for all $l$, $$\|a_{j,l}\|_{L^{\infty}(\mathbb{R}^d)}\lesssim 2^j,\,\,\|\nabla a_{j,l}\|_{L^{\infty}(\mathbb{R}^d)}\lesssim 1,$$and for $s\in (0,\,2^j]$ and $x\in {\rm supp}\,Q_sh_{s,j,l}$,
$$T_{\Omega,\,a}^{j}h(x)=a_{j,l}(x)W_{\Omega,j}h(x)-W_{\Omega,j}(a_{j,l}h)(x).
$$
For each fixed $j$ and $l$, let
$${\rm D}_{j,l,1}(s,t)=-\int_{\mathbb{R}^d}[a_{j,l},Q_s]W_{\Omega,j}Q_t^4f(x)Q_sh_{s,j,l}(x)dx,$$
$${\rm D}_{j,l,2}(s,t)=\int_{\mathbb{R}^d}a_{j,l}(x)Q_sW_{\Omega,j}Q_t^4f(x)Q_sh_{s,j,l}(x)dx,$$
$${\rm D}_{j,l,3}(s,t)=\int_{\mathbb{R}^d}Q_sW_{\Omega,\,j}[a_{j,l},Q_s]Q_t^4f(x)h_{s,j,l}(x)dx,$$
and
$${\rm D}_{j,l,4}(s,\,t)=-\int_{\mathbb{R}^d}Q_sW_{\Omega,\,j}(a_{j,l}Q_sQ_t^4f)(x)h_{s,j,l}(x)dx,
$$
where and in the following, for a locally integrable function $b$ and an operator $U$, $[b,\,U]$ denotes the commutator of $U$ with symbol $b$, namely,
\begin{eqnarray}\label{equation2.commutatordef}
[b,\,U]h(x)=b(x)Uh(x)-U(bh)(x).
\end{eqnarray}
Observe that both of $Q_s$ and $W_{\Omega,j}$ are convolution operators and $Q_sW_{\Omega,j}=W_{\Omega,j}Q_s$.
For $j\in\mathbb{Z}$ and $s\in (0,\,2^j]$, we have that
\begin{eqnarray}\label{equation2.decomp}
\int_{\mathbb{R}^d}Q_s^4T_{\Omega,\,a}^{j}Q_t^4f(x)g(x)dx&=&\sum_l\int_{\mathbb{R}^d}Q_sT_{\Omega,\,a}^{j}
Q_t^4f(x)Q_sh_{s,j,l}(x)dx\\
&=&\sum_{n=1}^4\sum_{l}{\rm D}_{j,l,n}(s,t).\nonumber
\end{eqnarray}
It now follows from H\"older's inequality that\begin{eqnarray*}
&&\Big|\sum_{j}\sum_l\int_{2^j}^{\infty}\int_{0}^{(2^jt^{\alpha-1})^{1/\alpha}}{\rm D}_{j,l,1}(s,t)\frac{ds}{s}\frac{dt}{t}
\Big|\\
&&\quad\le \Big\|\Big(\sum_j\sum_l\int_{2^j}^{\infty}\int_{0}^{(2^jt^{\alpha-1})^{1/\alpha}}|\chi_{I_{j,l}^*}Q_t^4f|^2
2^{-j}s
\frac{ds}{s}\frac{dt}{t}\Big)^{\frac{1}{2}}\Big\|_{L^2(\mathbb{R}^d)}\\
&&\quad\times\Big\|\Big(\sum_j\sum_l\int_{2^j}^{\infty}\int_{0}^{(2^jt^{\alpha-1})^{1/\alpha}}
|W_{\Omega,j}[a_{j,l},Q_s]Q_sh_{s,j,l}
|^2\frac{1}{2^{-j}s}
\frac{ds}{s}\frac{dt}{t}\Big)^{\frac{1}{2}}\Big\|_{L^{2}(\mathbb{R}^d)}.
\end{eqnarray*}
Invoking the fact that $\sum_{l}\chi_{I_{j,l}^*}\lesssim 1$, we deduce that
\begin{eqnarray*}&& \Big\|\Big(\sum_j\sum_l\int_{2^j}^{\infty}\int_{0}^{(2^jt^{\alpha-1})^{1/\alpha}}|\chi_{I_{j,l}^*}Q_t^4f|^2
2^{-j}s
\frac{ds}{s}\frac{dt}{t}\Big)^{\frac{1}{2}}\Big\|_{L^2(\mathbb{R}^d)}\\
&&\quad\lesssim \Big\|\Big(
\int^{\infty}_{0}|Q_t^4f|^2\int_{0}^{t}\sum_{j:\, 2^j\geq s^{\alpha}t^{1-\alpha}}2^{-j}s
\frac{ds}{s}
\frac{dt}{t}\Big)^{1/2}\Big\|_{L^2(\mathbb{R}^d)}\lesssim\|f\|_{L^2(\mathbb{R}^d)}.
\end{eqnarray*}
Let $M_{\Omega}$ be the operator defined by
$$M_{\Omega}f(x)=\sup_{r>0}r^{-d}\int_{|x-y|<r}|\Omega(x-y)||f(y)|dy.$$
The method of rotation of Calder\'on and Zygmund states that
\begin{eqnarray}\label{equation2.rotation}\|M_{\Omega}f\|_{L^p(\mathbb{R}^d)}\lesssim \|\Omega\|_{L^1(S^{d-1})}\|f\|_{L^p(\mathbb{R}^d)},\,\,p\in (1,\,\infty).\end{eqnarray}
Let $M$ be the Hardy-Littlewood maximal operator. Observe that when $s\in (0,\,2^j]$,
$$\big|[a_{j,l},\,Q_s]h(x)\big|\le \int_{\mathbb{R}^d}|\psi_s(x-y)||a_{j,l}(x)-a_{j,l}(y)||h(y)|dy\lesssim sMh(x).
$$
This, together with  (\ref{equation2.rotation}), yields
\begin{eqnarray*}
&&\Big\|\Big(\sum_j\sum_l\int_{2^j}^{\infty}\int_{0}^{(2^jt^{\alpha-1})^{1/\alpha}}|W_{\Omega,j}
[a_{j,l},Q_s]Q_sh_{s,j,l}
|^2(2^{-j}s)^{-1}
\frac{ds}{s}\frac{dt}{t}\Big)^{\frac{1}{2}}\Big\|^2_{L^{2}(\mathbb{R}^d)}\\
&&\quad\lesssim \sum_j\sum_l\int_{2^j}^{\infty}\int_{0}^{(2^jt^{\alpha-1})^{1/\alpha}}
\|M_{\Omega}MQ_sh_{s,j,l}
\|^2_{L^2(\mathbb{R}^d)}2^{-j}s
\frac{ds}{s}\frac{dt}{t}\\
&&\quad\lesssim \sum_j\sum_l\int_{2^j}^{\infty}\int_{0}^{(2^jt^{\alpha-1})^{1/\alpha}}
\|h_{s,j,l}
\|^2_{L^2(\mathbb{R}^d)}2^{-j}s
\frac{ds}{s}\frac{dt}{t}\lesssim\|g\|^2_{L^{2}(\mathbb{R}^d)},
\end{eqnarray*}
where  the last inequality follows from the fact that
\begin{eqnarray*}
\int_{s}^{\infty}\sum_{j:2^j\geq s^{\alpha}t^{1-\alpha}}2^{-j}s\frac{dt}{t}\lesssim 1.
\end{eqnarray*}
Therefore,
\begin{eqnarray}\label{equ2.13}
&&\Big|\sum_{j}\sum_l\int_{2^j}^{\infty}\int_{0}^{(2^jt^{\alpha-1})^{\frac{1}{\alpha}}}{\rm D}_{j,l,1}(s,t)\frac{ds}{s}\frac{dt}{t}
\Big|\lesssim \|f\|_{L^2(\mathbb{R}^d)}\|g\|_{L^{2}(\mathbb{R}^d)}.
\end{eqnarray}

Similar to the estimate (\ref{equ2.13}), we have that
\begin{eqnarray}\label{equ2.21}
&&\Big|\sum_{j}\sum_l\int_{2^j}^{\infty}\int_{0}^{(2^jt^{\alpha-1})^{\frac{1}{\alpha}}}{\rm D}_{j,l,3}(s,t)\frac{ds}{s}\frac{dt}{t}
\Big|\lesssim \|f\|_{L^2(\mathbb{R}^d)}\|g\|_{L^{2}(\mathbb{R}^d)}.\end{eqnarray}

To estimate the term  $\int_{2^j}^{\infty}\int_{0}^{(2^jt^{\alpha-1})^{1/\alpha}}{\rm D}_{j,l,2}(s,t)\frac{ds}{s}\frac{dt}{t}$, we write
\begin{eqnarray*}
{\rm D}_{j,l,2}(s,t)&=&\int_{\mathbb{R}^d}Q_sW_{\Omega,j}Q_t^4f(x)[a_{j,l},Q_s]h_{s,j,l}(x)dx\\
&+&\int_{\mathbb{R}^d}Q_sW_{\Omega,j}Q_t^4f(x)Q_s(a_{j,l}h_{s,j,l})(x)dx={\rm D}_{j,l,2}^1(s,t)+{\rm D}_{j,l,2}^2(s,t).
\end{eqnarray*}
Repeating the estimate for ${\rm D}_{j,l,1}$, we have that
\begin{eqnarray}\label{equ2.14}
&&\Big|\sum_{j}\sum_l\int_{2^j}^{\infty}\int_{0}^{(2^jt^{\alpha-1})^{\frac{1}{\alpha}}}{\rm D}_{j,l,2}^1(s,t)\frac{ds}{s}\frac{dt}{t}
\Big|\lesssim \|f\|_{L^2(\mathbb{R}^d)}\|g\|_{L^{2}(\mathbb{R}^d)}.
\end{eqnarray}
Write
\begin{eqnarray*}
&&\Big|\sum_{j}\sum_l\int_{2^j}^{\infty}\int_{0}^{(2^jt^{\alpha-1})^{\frac{1}{\alpha}}}{\rm D}_{j,l,2}^2(s,t)\frac{ds}{s}\frac{dt}{t}
\Big|\\
&&\quad\le
\Big\|\Big(\sum_{j}\int_{2^j}^{\infty}\int_{0}^{(2^jt^{\alpha-1})^{\frac{1}{\alpha}}}
|Q_s^2(2^jW_{\Omega,j})Q_t^3f|^2\log^{\sigma}(2^j/s+1)
\frac{ds}{s}\frac{dt}{t}\Big)^{\frac{1}{2}}
\Big\|_{L^2(\mathbb{R}^d)}\\
&&\quad\times
\Big\|\Big(\sum_{j}\int_{2^j}^{\infty}\int_{0}^{(2^jt^{\alpha-1})^{\frac{1}{\alpha}}}
\big|2^{-j}Q_t\big(\sum_la_{j,l}h_{s,j,l}\big)\big|^2
\log^{-\sigma}(\frac{2^j}{s}+1)\frac{ds}{s}\frac{dt}{t}\Big)^{\frac{1}{2}}
\Big\|_{L^{2}(\mathbb{R}^d)}\\
&&\quad:={\rm I}_1{\rm I}_2,
\end{eqnarray*}
where $\sigma>1$ is a constant such that $2\beta-\sigma>1$. Invoking the estimate (\ref{equa2.little}), we obtain that
\begin{eqnarray*}
\quad\quad{\rm I}_2&\lesssim& \Big(\sum_{j}\int_{0}^{2^j}
\big\|2^{-j}\sum_la_{j,l}h_{s,j,l}\big\|^2_{L^2(\mathbb{R}^d)}
\log^{-\sigma}(2^j/s+1)\frac{ds}{s}\Big)^{\frac{1}{2}}\\
&\lesssim&\Big(\sum_{j}\int_{0}^{2^j}
\big\|\sum_l|h_{s,j,l}|\big\|_{L^2(\mathbb{R}^d)}^2
\log^{-\sigma}(2^j/s+1)\frac{ds}{s}\Big)^{\frac{1}{2}}\nonumber\\
&=&\Big(\int_{0}^{\infty}
\|Q_s^2g\|^2_{L^2(\mathbb{R}^d)}
\sum_{j:2^j\geq s}\log^{-\sigma}(2^j/s+1)\frac{ds}{s}\Big)^{\frac{1}{2}}\lesssim \|g\|_{L^{2}(\mathbb{R}^d)}.\nonumber
\end{eqnarray*}
Note that $Q_s^2(2^jW_{\Omega,j})=Q_s(2^jW_{\Omega,j})Q_s$.  It follows  from Lemma \ref{lem2.3} and (\ref{equa2.little}) that
\begin{eqnarray*}
{\rm I}_1&=&\Big(\sum_{j}\int_{2^j}^{\infty}\int_{0}^{(2^jt^{\alpha-1})^{\frac{1}{\alpha}}}
\|Q_s^2(2^jW_{\Omega,j})Q_t^3f\|^2_{L^2(\mathbb{R}^d)}\log^{\sigma}(\frac{2^j}{s}+1)
\frac{ds}{s}\frac{dt}{t}\Big)^{\frac{1}{2}}\\
&\lesssim&\Big(\sum_{j}\int_{2^j}^{\infty}\int_{0}^{(2^jt^{\alpha-1})^{\frac{1}{\alpha}}}
\|Q_sQ_t^3f\|^2_{L^2(\mathbb{R}^d)}\log^{-2\beta+\sigma}(2^j/s+1)
\frac{ds}{s}\frac{dt}{t}\Big)^{\frac{1}{2}}\nonumber\\
&\lesssim&\Big\|\Big(\int_{0}^{\infty}\int_{0}^{\infty}
|Q_sQ_t^3f|^2
\frac{ds}{s}\frac{dt}{t}\Big)^{\frac{1}{2}}
\Big\|^2_{L^2(\mathbb{R}^d)}\lesssim \|f\|_{L^2(\mathbb{R}^d)}.\nonumber
\end{eqnarray*}
The estimates for ${\rm I}_1$ and ${\rm I}_2$ show that
\begin{eqnarray*}
&&\Big|\sum_{j}\sum_l\int_{2^j}^{\infty}\int_{0}^{(2^jt^{\alpha-1})^{1/\alpha}}{\rm D}_{j,l,2}^2(s,t)\frac{ds}{s}\frac{dt}{t}
\Big|\lesssim \|f\|_{L^2(\mathbb{R}^d)}\|g\|_{L^{2}(\mathbb{R}^d)}.
\end{eqnarray*}
This, together with (\ref{equ2.14}), gives us that
\begin{eqnarray}\label{equ2.20}
&&\Big|\sum_{j}\sum_l\int_{2^j}^{\infty}\int_{0}^{(2^jt^{\alpha-1})^{1/\alpha}}{\rm D}_{j,l,2}(s,t)\frac{ds}{s}\frac{dt}{t}
\Big|\lesssim \|f\|_{L^2(\mathbb{R}^d)}\|g\|_{L^{2}(\mathbb{R}^d)}.
\end{eqnarray}

We now estimate term corresponding to $\int_{2^j}^{\infty}\int_{0}^{(2^jt^{\alpha-1})^{1/\alpha}}{\rm D}_{j,l,4}(s,t)\frac{ds}{s}\frac{dt}{t}$. Write
\begin{eqnarray*}
&&\Big|\sum_{j}\sum_l\int_{2^j}^{\infty}\int_{0}^{(2^jt^{\alpha-1})^{1/\alpha}}{\rm D}_{j,l,4}(s,t)\frac{ds}{s}\frac{dt}{t}\Big |\\
&&\quad\leq
\Big\|\Big(\sum_j\int_{2^j}^{\infty}\int_{0}^{(2^jt^{\alpha-1})^{1/\alpha}}|Q_sQ_t^3f|^2\log^{-\sigma}(2^j/s+1)\frac{ds}{s}\frac{dt}{t}\Big)^{\frac{1}{2}}
\Big\|_{L^2(\mathbb{R}^d)}\\
&&\quad\times
\Big\|\Big(\sum_j\int_{2^j}^{\infty}\int_{0}^{(2^jt^{\alpha-1})^{1/\alpha}}\Big|Q_t\Big(\sum_la_{j,l}
W_{\Omega,j}Q_sh_{s,j,l}\Big)\Big|^2\log^{\sigma} (\frac{2^j}{s}+1)\frac{ds}{s}\frac{dt}{t}\Big)^{\frac{1}{2}}
\Big\|_{L^{2}(\mathbb{R}^d)}\\
&&\quad:={\rm I}_3{\rm I_4}.
\end{eqnarray*}
Obviously,
$${\rm I}_3\lesssim \Big\|\Big(\int_{0}^{\infty}\int_{0}^{\infty}|Q_sQ_t^3f|^2\frac{ds}{s}\frac{dt}{t}\Big)^{\frac{1}{2}}
\Big\|_{L^2(\mathbb{R}^d)}\lesssim\|f\|_{L^2(\mathbb{R}^d)}.
$$
On the other hand, it follows from Littlewood-Paley theory and Lemma \ref{lem2.3} that
\begin{eqnarray*}
{\rm I}_4&\lesssim&\Big(\sum_j\int_{0}^{2^j}\Big\|\sum_la_{j,l}
W_{\Omega,\,j}Q_sh_{s,j,l}\Big\|^2_{L^2(\mathbb{R}^d)}\log^{\sigma} (\frac{2^j}{s}+1)\frac{ds}{s}\Big)^{\frac{1}{2}}\\
&\lesssim &\Big(\sum_j\int_0^{2^j}2^{2j}\sum_l\|
W_{\Omega,j}Q_sh_{s,j,l}\|^2_{L^2(\mathbb{R}^d)}\log^{\sigma} (\frac{2^j}{s}+1)\frac{ds}{s}\Big)^{\frac{1}{2}}\\
&\lesssim &\Big(\sum_j\int_0^{2^j}\sum_l\|
h_{s,j,l}\|^2_{L^2(\mathbb{R}^d)}\log^{-2\beta+\sigma} (\frac{2^j}{s}+1)\frac{ds}{s}\Big)^{\frac{1}{2}}\lesssim \|g\|_{L^2(\mathbb{R}^d)},
\end{eqnarray*}
since $\|a_{j,l}\|_{L^{\infty}(\mathbb{R}^d)}\lesssim 2^j$, and the supports of functions $\{a_{j,l}
W_{\Omega,\,j}Q_sh_{s,j,l}\}$ have bounded overlaps.
The estimate for ${\rm I}_4$, together with the estimate for ${\rm I}_3$, gives us that
\begin{eqnarray}\label{equation2.21}\Big|\sum_{j}\sum_l\int_{2^j}^{\infty}\int_{0}^{(2^jt^{\alpha-1})^{1/\alpha}}{\rm D}_{j,l,4}(s,t)\frac{ds}{s}\frac{dt}{t}\Big |\lesssim \|f\|_{L^2(\mathbb{R}^d)}\|g\|_{L^{2}(\mathbb{R}^d)}.\end{eqnarray}
Combining inequalities (\ref{equ2.13}), (\ref{equ2.21}), (\ref{equ2.20}) and (\ref{equation2.21}) leads to that
$$|{\rm D}_2|\lesssim \|f\|_{L^2(\mathbb{R}^d)}\|g\|_{L^{2}(\mathbb{R}^d)}.$$

The estimate for ${\rm D}_1$ is fairly similar to the estimate ${\rm D}_2$. For example, since
$$\int_{0}^{t}\sum_{j:\, 2^j\geq t}2^{-j}s\frac{ds}{s}\lesssim 1, \,\,
\int_{s}^{\infty}\sum_{j:2^j\geq t}2^{-j}s\frac{dt}{t}\lesssim 1,
$$
we have that
\begin{eqnarray*}
&&\Big|\sum_{j}\sum_l\int_0^{2^j}\int_{0}^{t}{\rm D}_{j,l,1}(s,t)\frac{ds}{s}\frac{dt}{t}
\Big|\\
&&\quad\le \Big\|\Big(\sum_j\sum_l\int_0^{2^j}\int_{0}^{t}|\chi_{I_{j,l}^*}Q_t^4f|^2
2^{-j}s
\frac{ds}{s}\frac{dt}{t}\Big)^{\frac{1}{2}}\Big\|_{L^2(\mathbb{R}^d)}\\
&&\quad\times\Big\|\Big(\sum_j\sum_l\int_0^{2^j}\int_{0}^{\infty}
|W_{\Omega,j}[a_{j,l},Q_s]Q_sh_{s,j,l}
|^2(2^{-j}s)^{-1}
\frac{ds}{s}\frac{dt}{t}\Big)^{\frac{1}{2}}\Big\|_{L^{2}(\mathbb{R}^d)}\\
&&\quad\lesssim \|f\|_{L^2(\mathbb{R}^d)}\|g\|_{L^{2}(\mathbb{R}^d)}.
\end{eqnarray*}
The estimates for terms $\sum_{j}\sum_l\int_0^{2^j}\int_{0}^{t}{\rm D}_{j,l,i}(s,t)\frac{ds}{s}\frac{dt}{t}$ $(i=2,3,4)$ are parallel to the estimates for
$\sum_{j}\sum_l\int_{2^j}^{\infty}\int_{0}^{(2^jt^{\alpha-1})^{1/\alpha}}{\rm D}_{j,l,i}(s,t)\frac{ds}{s}\frac{dt}{t}$. Altogether, we have that
$$|{\rm D}_1|\lesssim \|f\|_{L^2(\mathbb{R}^d)}\|g\|_{L^{2}(\mathbb{R}^d)}.$$

It remains to consider ${\rm D}_3$. This was essentially proved in \cite[pp.\,1281-1283]{hof1}. For the sake of self-contained, we present the details here. Set
$$ h(x,y)=\int\int\psi_s(x-z)\sum_{j:2^j\leq s^\alpha t^{1-\alpha}}K_j(z,u)[\psi_t(u-y)-\psi_t(x-y)]dudz.$$
Let $H$ be the operator with integral kernel $h$.
It then follows that
\begin{eqnarray*}
|{\rm D}_3|&\lesssim &\Big|\int_0^\infty \int_0^t\int_{\mathbb{R}^d} HQ_t^3f(x)Q_s^3g(x)dx\frac{ds}{s}\frac{dt}{t}\Big|\\
&+&\Big|\sum_{j\in\mathbb{Z}}\int^{\infty}_{2^j}\int_{(2^jt^{\alpha-1})^{\frac{1}{\alpha}}}^t
\int_{\mathbb{R}^d}(Q_s{T}_{\Omega,\,a}^{j}1)(x)Q_t^4f(x)Q_s^3g(x)dx\frac{ds}{s}\frac{dt}{t}\Big|\nonumber\\
&=&|{\rm D}_{31}|+|{\rm D}_{32}|.\nonumber
\end{eqnarray*}
As in \cite[p.\,1282]{hof1}, we obtain by Lemma \ref{lem2.2} and the mean value theorem that
$$|h(x,y)|\lesssim\Big(\frac{s}{t}\Big)^{\varrho} t^{-d}\chi_{\{(x,y):|x-y|\leq Ct\}}(x,y),$$
where $\varrho =(d+2)\alpha-d-1\in (0,\,1)$. Then we  have
$$|HQ_t^3f(x)|\lesssim (\frac{s}{t})^{\varrho} M(Q_t^3f)(x),$$
and
\begin{eqnarray*}
|{\rm D}_{31}|&\lesssim &\int_0^\infty \int_0^t\int_{\mathbb{R}^d}|M(Q_t^3f)(x)||Q_s^3g(x)|dx\big(\frac{s}{t}\big)^{\varrho}\frac{ds}{s}\frac{dt}{t}\\
&\lesssim&\Big\|\Big(\int_0^\infty \int_0^t|M(Q_t^3f)|^{2}(\frac{s}{t})^{\varrho} \frac{ds}{s}\frac{dt}{t}\Big)^{\frac{1}{2}}\Big\|_{L^2(\mathbb{R}^d)}\nonumber\\
&&\times \Big\|\Big(\int_0^\infty\int_s^\infty|Q_s^3g|^{2}(\frac{s}{t})^{\varrho}\frac{dt}{t}\frac{ds}{s}
\Big)^{\frac{1}{2}}\Big\|_{L^{2}(\mathbb{R}^d)}\nonumber\\
&\lesssim&\Big\|\Big(\int_0^\infty|M(Q_t^3f)|^2\frac{dt}{t}\Big)^{\frac{1}{2}}\Big\|_{L^2(\mathbb{R}^d)}
\Big\|\Big(\int_0^\infty|Q_s^3g|^{2}
\frac{ds}{s}
\Big)^{\frac{1}{2}}\Big\|_{L^{2}(\mathbb{R}^d)}\nonumber\\
&\lesssim&\|f\|_{L^2(\mathbb{R}^d)}\|g\|_{L^2(\mathbb{R}^d)}.\nonumber
\end{eqnarray*}

As for ${\rm D}_{32}$, we split it into three parts as follows:
\begin{eqnarray*}
{\rm D}_{32}=\sum_{j\in\mathbb{Z}}\int^{\infty}_{0}\int_{0}^t-
\sum_{j\in\mathbb{Z}}\int^{2^j}_{0}\int_{0}^t-\sum_{j\in\mathbb{Z}}
\int_{2^j}^\infty\int_{0}^{(2^jt^{\alpha-1})^{\frac{1}{\alpha}}}={\rm D}_{321}-{\rm D}_{322}-{\rm D}_{323}.
\end{eqnarray*}
Let
$$\zeta(x)=\int^{\infty}_1\psi_t*\psi_t*\psi_t*\psi_t(x)\frac{dt}{t},\,\,P_s=\int_s^\infty Q_t^4\frac{dt}{t}.$$
Han and Sawyer \cite{han} proved that $\zeta$ is a radial function which is supported on a ball having radius $C$ and has mean value zero. Observe that $P_sf(x)=\zeta_s*f(x)$ with $\zeta_s(x)=s^{-d}\zeta(s^{-1}x)$.
The Littlewood-Paley theory tells us that
$$\Big\|\Big(\int^{\infty}_0|P_sf|^2\frac{ds}{s}\Big)^{\frac{1}{2}}\Big\|_{L^2(\mathbb{R}^d)}\lesssim \|f\|_{L^2(\mathbb{R}^d)}.
$$
(i) of Lemma \ref{lem2.4} states that $T_{\Omega,a}1\in {\rm  BMO}(\mathbb{R}^d)$. Recall that ${\rm supp}\,\psi\subset B(0,\,1)$ and $\psi$ has integral zero. Thus  for $x\in \mathbb{R}^d$,
\begin{eqnarray*}
|Q_s(T_{\Omega,a}1)(x)|\le s^{-d}\int_{|x-y|\leq s}|\psi(s^{-1}(x-y))||T_{\Omega,a}1(y)-\langle T_{\Omega,a}1\rangle_{B(x,s)}|dy\lesssim 1,
\end{eqnarray*}
where $\langle T_{\Omega,a}1\rangle_{B(x,s)}$ denotes the mean value of $T_{\Omega,a}1$ on the ball centered at $x$ and having radius $s$.
Therefore,
\begin{eqnarray*}
|{\rm D}_{321}|&=&\Big|\int_{\mathbb{R}^d}\int_0^\infty Q_sT_{\Omega,a}1(x)P_sf(x)Q_s^3g(x)\frac{ds}{s}dx\Big|\\
&\lesssim &\Big\|\Big(\int^{\infty}_0|P_sf|^2\frac{ds}{s}\Big)^{\frac{1}{2}}\Big\|_{L^2(\mathbb{R}^d)}
\Big\|\Big(\int^{\infty}_0|Q_s^3g|^{2}\frac{ds}{s}\Big)^{\frac{1}{2}}\Big\|_{L^{2}(\mathbb{R}^d)}\\
&\lesssim&\|f\|_{L^2(\mathbb{R}^d)}\|g\|_{L^2(\mathbb{R}^d)}.
\end{eqnarray*}
From (ii) of Lemma \ref{lem2.4} and H\"older's inequality, we obtain that
\begin{eqnarray*}
|{\rm D}_{322}|&\lesssim & \Big\|\Big(\sum_j\int_{0}^{2^j}\int_0^t 2^{-j}s|Q_t^4f|^2\frac{ds}{s}\frac{dt}{t}\Big)^{1/2}\Big\|_{L^2(\mathbb{R}^d)}\\
&&\times\Big\|\Big(\sum_j\int_0^{2^j}\int_0^t 2^{-j}s|Q_s^3g|^{2}\frac{ds}{s}\frac{dt}{t}\Big)^{1/2}\Big\|_{L^2(\mathbb{R}^d)}\\
&\lesssim &\|f\|_{L^2(\mathbb{R}^d)}\|g\|_{L^2(\mathbb{R}^d)}.
\end{eqnarray*}
The same result holds true for ${\rm D}_{323}$. Combining the estimates for terms ${\rm D}_{321}$, ${\rm D}_{322}$ and ${\rm D}_{323}$ give us that
$$|{\rm D}_3|\lesssim \|f\|_{L^2(\mathbb{R}^d)}\|g\|_{L^2(\mathbb{R}^d)}.$$
This leads to (\ref{eq2.x1}) and then establishes  the $L^2(\mathbb{R}^d)$ boundedness of $T_{\Omega,a}$.\qed

\section{Proof of Theorem \ref{thm1.1}: $L^p$ boundedness}

We begin with some lemmas.

\begin{lemma}\label{lem0.1}Let $\varpi\in C^{\infty}_0(\mathbb{R}^d)$  be a radial function such that ${\rm supp}\, \varpi\subset\{1/4\leq |\xi|\leq 4\}$  and
$$\sum_{l\in\mathbb{Z}}\varpi^3(2^{-l}\xi)=1,\,\,\,|\xi|>0,$$ and $S_l$ be the multiplier operator defined by
$$\widehat{S_lf}(\xi)=\varpi(2^{-l}\xi)\widehat{f}(\xi).$$
Let $k\in\mathbb{Z}_+$, $a$ be a function on $\mathbb{R}^d$ such that $\nabla a\in L^{\infty}(\mathbb{R}^d)$. Then
\begin{eqnarray}\label{equation01}
\Big\|\Big(\sum_{l\in\mathbb{Z}}\big|2^{kl}[a,\,S_l]^kf\big|^2\Big)^{\frac{1}{2}}\Big\|_{L^2(\mathbb{R}^d)}\lesssim \|f\|_{L^2(\mathbb{R}^d)},
\end{eqnarray}
and
\begin{eqnarray}\label{equation02}
\Big\|\sum_{l\in\mathbb{Z}}2^{kl}[a,\,S_l]^kf_l\Big\|_{L^2(\mathbb{R}^d)}\lesssim \Big\|\Big(\sum_{l}|f_l|^2\Big)^{1/2}\Big\|_{L^2(\mathbb{R}^d)},
\end{eqnarray}
where and in the following, for a locally integrable function $a$ and an operator $U$, $[a,\,U]^0f=Uf$, while for $k\in\mathbb{N}$ with $k\geq 1$, $[a,\,U]^k$ denotes the commutator of $[a,\,U]^{k-1}$ and $a$, defined as  (\ref{equation2.commutatordef}).
\end{lemma}
Note that (\ref{equation02}) follows from (\ref{equation01}) and a duality argument.  For the case of $k=0$, (\ref{equation01}) follows from Littlewood-Paley theory. Inequality (\ref{equation01}) with $k=1$ was proved in \cite[Lemma 2.3]{cdhx}, while for the case of $k\geq 2$, the proof of (\ref{equation01}) is similar to the proof of \cite[Lemma 2.3]{cdhx}.

\begin{lemma}\label{lem0.2} Let $k\in\mathbb{N}$, $n\in\mathbb{Z}_+$ with $n\leq k$, $D$, $E$ be positive constants and $E\leq 1$,  $m$ be a  multipliers
such that $m\in L^1(\mathbb{R}^d)$, and
$$\|m\|_{L^{\infty}(\mathbb{R}^d)}\le D^{-k}E$$
and for all multi-indices $\gamma\in\mathbb{Z}_+^d$,
$$\|\partial^{\gamma}m\|_{L^{\infty}(\mathbb{R}^d)}\le D^{|\gamma|-k}.
$$
Let $a$ be a function on $\mathbb{R}^d$ with $\nabla a\in L^{\infty}(\mathbb{R}^d)$, and $T_m$ be the multiplier operator defined by
$$\widehat{T_{m}f}(\xi) = m(\xi)\widehat{f}(\xi).$$
Then for any $\varepsilon\in (0,\,1)$,
$$\|[a,\,T_{m}]^nf\|_{L^2(\mathbb{R}^d)}\lesssim D^{n-k}E^{\varepsilon}\|f\|_{L^2(\mathbb{R}^d)}.
$$
\end{lemma}

\begin{proof} Our argument here is a generalization  of the proof of Lemma 2 in \cite{huguoen}, together with some more refined estimates,   see also \cite[lemma 2.3]{mahu} for the original version.    We only consider the case $1\leq n\leq k$, since
$$\|[a,\,T_{m}]^0f\|_{L^2(\mathbb{R}^d)}\lesssim D^{-k}E^{\varepsilon}\|f\|_{L^2(\mathbb{R}^d)}
$$
holds obviously.

Let $\varphi\in C^{\infty}_0(\mathbb{R}^d)$ be the same  as in (\ref{eq2.varphi}). Recall that ${\rm supp}\,\varphi\subset \{1/4\leq |x|\leq 4\}$, and
$$\sum_{j\in\mathbb{Z}}\varphi(2^{-j}x)\equiv 1,\,\,|x|>0.$$
Let  $\varphi_{l,D}(x)=\varphi(2^{-l}D^{-1}x)$ for $l\in\mathbb{Z}$. Set
$$W_{l}(x)=K(x)\varphi_{l,D}(x),\,\,l\in\mathbb{Z},$$
where $K$ is the inverse Fourier transform of $m$.
Observing that for all multi-indices $\gamma\in\mathbb{Z}_+^d$, $\partial^{\gamma}\varphi(0)=0$,  we thus have that
$$\int_{\mathbb{R}^d}\widehat{\varphi}(\xi)\xi^{\gamma}d\xi=0.$$
This, in turn, implies that for all $N\in \mathbb{N}$ and $\xi\in\mathbb{R}^d$,
\begin{eqnarray}\label{equation03}
|\widehat{W_l}(\xi)|&=&\Big|\int_{\mathbb{R}^d}\Big(m(\xi-\frac{\eta}{2^lD})-\sum_{|\gamma|\leq N}\frac{1}{\gamma!}\partial^{\gamma}m(\xi)(\frac{\eta}{2^lD})^{\gamma}\Big)\widehat{\varphi}(\eta)d\eta\Big|\\
&\lesssim&2^{-l(N+1)}D^{-(N+1)}\sum_{|\gamma|=N+1}\|\partial^{\gamma}m\|_{L^{\infty}(\mathbb{R}^d)}
\int_{\mathbb{R}^d}|\eta|^{N+1}|\widehat{\varphi}(\eta)|d\eta\nonumber\\
&\lesssim &2^{-l(N+1)}D^{-k}.\nonumber
\end{eqnarray}
On the other hand, a trivial computation gives that for $l\in\mathbb{Z}$,
\begin{eqnarray}\label{equation04}
\|\widehat{W_l}\|_{L^{\infty}(\mathbb{R}^d)}\leq \|m\|_{L^{\infty}(\mathbb{R}^d)}\|\widehat{\varphi_{l,D}}\|_{L^1(\mathbb{R}^d)}\lesssim D^{-k}E.
\end{eqnarray}
Combining the inequalities (\ref{equation03}) and (\ref{equation04}) shows that for any $l\in\mathbb{Z}$, $N\in\mathbb{N}$ and $\varepsilon\in (0,\,1)$,
\begin{eqnarray}\label{equation05}
&&\|\widehat{W_l}\|_{L^{\infty}(\mathbb{R}^d)}\lesssim 2^{-l(N+1)(1-\varepsilon)}D^{-k}
E^{\varepsilon}.
\end{eqnarray}
Let $T_{m,l}$ be the convolution operator with kernel $W_l$. Inequality (\ref{equation05}), via Plancherel's theorem, tells us that for $l\in\mathbb{Z}$ and $N\in \mathbb{N}$,
\begin{eqnarray}\label{equation06}
\|T_{m,l}f\|_{L^2(\mathbb{R}^d)}\lesssim 2^{-l(N+1)(1-\varepsilon)}
D^{-k}
E^{\varepsilon}\|f\|_{L^2(\mathbb{R}^d)}.
\end{eqnarray}

We claim that for all $l\in\mathbb{Z}$, $N\in\mathbb{N}$ and $\varepsilon\in (0,\,1)$,
\begin{eqnarray}\label{equation07}
\|[a,\,T_{m,l}]^nf\|_{L^2(\mathbb{R}^d)}\lesssim 2^{-l(N+1)(1-\varepsilon)+ln}D^{n-k}E^{\varepsilon}\|f\|_{L^2(\mathbb{R}^d)}.
\end{eqnarray}
Observe that ${\rm supp}\,W_l\subset \{x:\,|x|\leq D2^{l+2}\}$. If $I$ is a cube having side length $2^lD$, and $f\in L^2(\mathbb{R}^d)$ with ${\rm supp}\,f\subset I$, then $T_{m,l}f\subset 100dI$. Therefore,  to prove (\ref{equation07}), we may assume that ${\rm supp}\, f\subset I$ with $I$ a cube having side length $2^lD$. Let $x_0\in I$ and $a_I(y)=(a(y)-a(x_0))\chi_{100dI}(y)$. Then
$$\|a_I\|_{L^{\infty}(\mathbb{R}^d)}\lesssim 2^lD.$$
Write
$$[a,\,T_{m,l}]^nf(x)=\sum_{i=0}^n(a_I(x))^iC_n^i
T_{m,l}\big((-a_I)^{k-i}f\big)(x).$$
It then follows from (\ref{equation06}) that
\begin{eqnarray*}\|[a,\,T_{m,\,l}]^nf\|_{L^2(\mathbb{R}^d)}&\lesssim &\sum_{i=0}^n2^{il}D^i\|T_{m,l}\big((-a_I)^{n-i}f\big)\|_{L^2(\mathbb{R}^d)}
\\
&\lesssim & 2^{nl-l(N+1)(1-\varepsilon)}D^{n-k}
E^{\varepsilon}\|f\|_{L^2(\mathbb{R}^d)}.
\end{eqnarray*}
This yields (\ref{equation07}).

We now conclude the proof of Lemma \ref{lem0.2}. Recall that $E\in (0,\,1]$. It suffices to prove Lemma \ref{lem0.2} for the case of $\varepsilon\in (2/3,\,1)$.
For fixed $\varepsilon\in (2/3,\,1)$, we choose $N_1\in\mathbb{N}$ such that $(N_1+1)(1-\varepsilon)>n$, $N_2\in\mathbb{N}$ such that $(N_2+1)(1-\varepsilon)<n$. It follows from (\ref{equation07}) that
\begin{eqnarray*}
\|[a,\,T_{m}]^nf\|_{L^2(\mathbb{R}^d)}&\leq&\sum_{l\leq 0}\|[a,\,T_{m,\,l}]^nf\|_{L^2(\mathbb{R}^d)} + \sum_{l\in\mathbb{N}}\|[a,\,T_{m,l}]^nf\|_{L^2(\mathbb{R}^d)}\\
&\lesssim&D^{n-k}E^{\varepsilon}\sum_{l\in\mathbb{N}} 2^{-l(N_1+1)(1-\varepsilon)+ln}\|f\|_{L^2(\mathbb{R}^d)}\\
&&+D^{n-k}E^{\varepsilon}\sum_{l\leq 0} 2^{-l(N_2+1)(1-\varepsilon)+ln}\|f\|_{L^2(\mathbb{R}^d)}\\
&\lesssim&D^{n-k}E^{\varepsilon}\|f\|_{L^2(\mathbb{R}^d)}.
\end{eqnarray*}
This completes the proof of Lemma \ref{lem0.2}.
\end{proof}
\begin{lemma}\label{lem0.3}Let $k\in\mathbb{N}$, $n\in\mathbb{Z}_+$ with $n\leq k$, $D$, $A$ and $B$ be positive constants with $A,\,B<1$,  $m$ be a  multipliers
such that $m\in L^1(\mathbb{R}^d)$, and
$$\|m\|_{L^{\infty}(\mathbb{R}^d)}\le D^{-k}(AB)^{k+1},$$
and for all multi-indices $\gamma\in\mathbb{Z}^d_+$,
$$\|\partial^{\gamma}m\|_{L^{\infty}(\mathbb{R}^d)}\le D^{|\gamma|-k}B^{-|\gamma|}.
$$
Let $T_m$ be the multiplier operator defined by
$$\widehat{T_{m}f}(\xi) = m(\xi)\widehat{f}(\xi).$$
Let $a$ be a function on $\mathbb{R}^d$ such that $\nabla a\in L^{\infty}(\mathbb{R}^d)$. Then for any $\sigma\in (0,\,1)$,
\begin{eqnarray}\label{equation0.summing}\big\|[a,\,T_{m}]^nf\big\|_{L^2(\mathbb{R}^d)}\lesssim D^{n-k}A^{\sigma}B^{k-n+\sigma}\|f\|_{L^2(\mathbb{R}^d)}.
\end{eqnarray}
\end{lemma}
\begin{proof} Let $T_{m,l}$ be the same as in the proof of Lemma \ref{lem0.2}. As in the proof of Lemma \ref{lem0.2},  we know that for all $l\in\mathbb{Z}$, $N\in\mathbb{N}$ and $\varepsilon\in (0,\,1)$,
\begin{eqnarray}\label{equation0.summing1}\big\|[a,\,T_{m,l}]^nf\big\|_{L^2(\mathbb{R}^d)}&\lesssim& 2^{-l(N+1)(1-\varepsilon)+nl}D^{n-k}\\
&&\times B^{-(N+1)(1-\varepsilon)+(k+1)\varepsilon}A^{(k+1)\varepsilon}\|f\|_{L^2(\mathbb{R}^d)}.\nonumber
\end{eqnarray}

For each fixed $\sigma\in (0,\,1)$, we choose $\varepsilon\in (0,\,1)$ such that
$$(k+1)\varepsilon-k-\sigma>1-\varepsilon,$$
and choose $N_1\in\mathbb{N}$ such that
$$(N_1+1)(1-\varepsilon)>n,\,\, -(N_1+1)(1-\varepsilon)+(k+1)\varepsilon>k-n+\sigma.$$
Also, we choose $N_2\in\mathbb{N}$ such that $(N_2+1)(1-\varepsilon)<n$. Note that such a $N_2$ satisfies $$-(N_2+1)(1-\varepsilon)+(k+1)\varepsilon>k-n+\sigma.$$
Recalling that $B<1$, we have that
$$B^{-(N_1+1)(1-\varepsilon)+(k+1)\varepsilon}\leq B^{k-n+\sigma},\,\,B^{-(N_2+1)(1-\varepsilon)+(k+1)\varepsilon}
\leq B^{k-n+\sigma}.$$
Our desired estimate (\ref{equation0.summing}) now follows (\ref{equation0.summing1}) by
\begin{eqnarray*}
\|[a,\,T_{m}]^nf\|_{L^2(\mathbb{R}^d)}
&\lesssim&D^{n-k}A^{\sigma}B^{k-n+\sigma}\sum_{l\in\mathbb{N}} 2^{-l(N_1+1)(1-\varepsilon)+ln}\|f\|_{L^2(\mathbb{R}^d)}\\
&&+D^{n-k}A^{\sigma}B^{k-n+\sigma}\sum_{l\leq 0} 2^{-l(N_2+1)(1-\varepsilon)+ln}\|f\|_{L^2(\mathbb{R}^d)}\\
&\lesssim&D^{n-k}B^{k-n+\sigma}A^{\sigma}\|f\|_{L^2(\mathbb{R}^d)},
\end{eqnarray*}
since $(k+1)\varepsilon>\sigma$ and $A<1$. This completes the proof of Lemma \ref{lem0.3}.
\end{proof}
The following conclusion is a variant of Theorem 1 in \cite{huguoen}, and will be useful in the proof of Theorem \ref{thm1.1}.
\begin{theorem}\label{thm0.21}
Let $k\in\mathbb{N}$, $A\in (0,\,1/2)$ be a constant, $\{\mu_j\}_{j\in\mathbb{Z}}$ be a sequence of functions on $\mathbb{R}^d\backslash \{0\}$. Suppose that for some $\beta\in (1,\,\infty)$,
$$\|\mu_j\|_{L^1(\mathbb{R}^d)}\lesssim 2^{-jk},\, |\widehat{\mu_j}(\xi)|\lesssim 2^{-jk}\min\{|A2^j\xi|^{k+1},\,\log ^{-\beta}(2+|2^{j}\xi|)\},$$
and for all multi-indices $\gamma\in\mathbb{Z}_+^d$,
$$\|\partial^{\gamma}\widehat{\mu_j}\|_{L^{\infty}(\mathbb{R}^d)}\lesssim 2^{j(|\gamma|-k)}.$$
Let $K(x)=\sum_{j\in\mathbb{Z}}\mu_j(x)$ and $T$ be the convolution operator with kernel $K$. Then for any $\varepsilon\in (0,\,1)$,  function $a$ on $\mathbb{R}^d$ with $\nabla a\in L^{\infty}(\mathbb{R}^d)$,
$$\|[a,\,T]^kf\|_{L^2(\mathbb{R}^d)}\lesssim \log ^{-\varepsilon\beta+1}\big(\frac{1}{A}\big)\|f\|_{L^2(\mathbb{R}^d)}.
$$
\end{theorem}
\begin{proof} At first, we claim that for $k_1\in \mathbb{Z}$ with $0\leq k_1\leq k$,
\begin{eqnarray}\label{equation0.fourier}
\|Tf\|_{L^2_{k_1-k}(\mathbb{R}^d)}\lesssim \|f\|_{L^2_{k_1}(\mathbb{R}^d)},
\end{eqnarray}
where $\|f\|_{L^2_{k_w}(\mathbb{R}^d)}$  for $k_2\in\mathbb{Z}$ is the Sobolev norm defined as
$$\|f\|_{L^2_{k_2}(\mathbb{R}^d)}^2=\int_{\mathbb{R}^d}|\xi|^{2k_2}|\widehat{f}(\xi)|^2d\xi.$$
In fact, by the Fourier transfrom estimate of $\mu_j$, we have that for each fixed $\xi\in\mathbb{R}^d\backslash \{0\}$,
$$\sum_{j\in\mathbb{Z}}|\widehat{\mu_j}(\xi)|\lesssim \sum_{j:\, 2^j\geq |\xi|^{-1}}2^{-jk}+|\xi|^{k+1}\sum_{j:\,2^j\leq |\xi|^{-1}}2^{j}\lesssim |\xi|^k.$$
This, together with Plancherel's theorem, gives (\ref{equation0.fourier}).

Let $U_j$ be the convolution operator with kernel $\mu_j$, and $\varpi\in C^{\infty}_0(\mathbb{R}^d)$ such that $0\leq \varpi\leq 1$, ${\rm supp}\,\varpi\subset \{1/4\leq |\xi|\leq 4\}$ and
$$\sum_{l\in\mathbb{Z}}\varpi^3(2^{-l}\xi)=1,\,\,|\xi|>0.$$
Set $m_j(x)=\widehat{\mu_j}(\xi)$,
and $m_j^l(\xi)=m_j(\xi)\varpi(2^{j-l}\xi)$. Define the operator $U_j^l$ by
$$\widehat{U_j^lf}(\xi)=m_j^l(\xi)\varpi(2^{j-l}\xi)\widehat{f}(\xi).$$
Now let  $S_l$ be the multiplier operator defined as in  Lemma \ref{lem0.1}. Let $f\in C^{\infty}_0(\mathbb{R}^d)$, $B=B(0,\,R)$ be a ball such that ${\rm supp}\,f\subset B$, and let $x_0\in B$. We can write
\begin{eqnarray}\label{equation0equiv}
\quad[a,\,T]^kf&=&\sum_{n=0}^kC_k^n(a-a(x_0))^{k-n}T\big((a(x_0)-a)^nf)(x)\\
&=&\sum_{n=0}^kC_k^n(a-a(x_0))^{k-n}\sum_{l}\sum_j(S_{l-j}U_{j}^lS_{l-j})\big((a(x_0)-a)^nf)\nonumber\\
&=&\sum_{l}\sum_j[a,\,S_{l-j}U_{j}^lS_{l-j}]^kf.\nonumber
\end{eqnarray}

We now estimate $
\big\|[a,\,S_{l-j}U_j^lS_{l-j}]^kf\big\|_{L^2(\mathbb{R}^d)}.
$
At first, we have that $m_j^l\in L^1(\mathbb{R}^d)$ and
\begin{eqnarray*}
|m_j^l(\xi)|\lesssim 2^{-jk}\min\{A^{k+1}2^{l(k+1)},\,\log^{-\beta} (2+2^l)\}.\end{eqnarray*}
Furthermore, by the fact that
$$|\partial^{\gamma}\phi(2^{j-l}\xi)|\lesssim 2^{(j-l)|\gamma|},\,\,|\partial^{\gamma}m_j(\xi)|\lesssim 2^{j(|\gamma|-k)},$$
it then follows that for all $\gamma\in\mathbb{Z}_+^d$,
$$|\partial^{\gamma}m_j^l(\xi)|\lesssim \Big \{\begin{array}{ll}2^{j(|\gamma|-k)}\,\,&\hbox{if}\,\,l\in\mathbb{N}\\
2^{j(|\gamma|-k)}2^{-|\gamma|l},\,\,&\hbox{if}\,\,l\leq 0.\end{array}
$$
An application of Lemma \ref{lem0.2} (with $D=2^j$, $E=\min\{(A2^l)^{k+1},\,l^{-\beta}\}$) yields
\begin{eqnarray}\label{equation0.l2estimate1}
\|[a,\,U_j^l]^nf\|_{L^2(\mathbb{R}^d)}\lesssim 2^{j(n-k)}\min\{(A2^l)^{k+1},\, l^{-\beta}\}^{\varepsilon}\|f\|_{L^2(\mathbb{R}^d)},\,\,l\in\mathbb{N}.
\end{eqnarray}
On the other hand, we deduce from Lemma \ref{lem0.3} (with $D=2^j$ and $B=2^l$) that for some $\sigma\in (0,\,1)$,
\begin{eqnarray}\label{equation0.l2estimate2}
\|[a,\,U_j^l]^nf\|_{L^2(\mathbb{R}^d)}\lesssim 2^{j(n-k)}2^{l(k-n)}A^{\sigma}2^{\sigma l}\|f\|_{L^2(\mathbb{R}^d)},\,\,l\leq 0.
\end{eqnarray}
Write
$$[a,\,S_{l-j}U_j^lS_{l-j}]^k=\sum_{n_1=0}^kC_{k}^{n_1}[a,\,S_{l-j}]^{n_1}\sum_{n_2=0}^{k-n_1}C_{k-n_1}^{n_2}
[a,\,U_j^l]^{n_2}[a,\,S_{l-j}]^{k-n_1-n_2}.
$$
For fixed $n_1,\,n_2,\,n_3\in\mathbb{Z}_+$ with $n_1+n_2+n_3=k$, a standard computation involving Lemma \ref{lem0.1}, estimates (\ref{equation0.l2estimate1}) and (\ref{equation0.l2estimate2})  leads to that for $l\in\mathbb{N}$,
\begin{eqnarray*}
&&\big\|\sum_{j\in\mathbb{Z}}[a,\,S_{l-j}]^{n_1}[a,\,U_j^l]^{n_2}[a,\,S_{l-j}]^{n_3}
f\big\|^2_{L^2(\mathbb{R}^d)}\\
&&\quad\lesssim\sum_{j\in\mathbb{Z}}2^{2(j-l)n_1}\|[a,\,U_{j}^l]^{n_2}[a,\,S_{l-j}]^{n_3}f\|^2_{L^2
(\mathbb{R}^d)}\\
&&\quad\lesssim\min\{(A2^l)^{k+1},\, l^{-\beta}\}^{2\varepsilon}
\|f\|^2_{L^2(\mathbb{R}^d)};
\end{eqnarray*}
and for $l\in\mathbb{Z}_-$,
\begin{eqnarray*}
&&\big\|\sum_{j\in\mathbb{Z}}[a,\,S_{l-j}]^{n_1}[a,\,U_j^l]^{n_2}[a,\,S_{l-j}]^{n_3}
f\big\|^2_{L^2(\mathbb{R}^d)}\\
&&\quad\lesssim\sum_{j\in\mathbb{Z}}2^{2(j-l)n_1}\|[a,\,U_{j}^l]^{n_2}[a,\,S_{l-j}]^{n_3}f\|^2_{L^2
(\mathbb{R}^d)}\\
&&\quad\lesssim A^{2\sigma}2^{2\sigma l}
\|f\|^2_{L^2(\mathbb{R}^d)}.
\end{eqnarray*}
 Therefore,
\begin{eqnarray*}
\sum_{l}\|[a,\,S_{l-j}U_j^lS_{l-j}]^kf\|_{L^2(\mathbb{R}^d)}&=&\sum_{l:\,l>\log (\frac{1}{\sqrt{A} })}\|[a,\,S_{l-j}U_j^lS_{l-j}]^kf\|_{L^2(\mathbb{R}^d)}\\
&&+\sum_{l:\,0\leq l\leq \log (\frac{1}{\sqrt{ A}})}\|[a,\,S_{l-j}U_j^lS_{l-j}]^kf\|_{L^2(\mathbb{R}^d)}\\
&&+\sum_{l:\,l<0}\|[a,\,S_{l-j}U_j^lS_{l-j}]^kf\|_{L^2(\mathbb{R}^d)}\\
&\lesssim&\Big(\sum_{l:\,l>\log (\frac{1}{\sqrt{ A}})}l^{-\varepsilon\beta}+A^{\sigma}\sum_{l:\,l<0}2^{\sigma l}\Big)\|f\|_{L^2(\mathbb{R}^d)}\\
&&+A^{(k+1)\varepsilon}\sum_{l:\,0\leq l\leq \log(\frac{1}{\sqrt{ A}})}2^{(k+1)l\varepsilon}\|f\|_{L^2(\mathbb{R}^d)}\\
&\lesssim&\log^{-\varepsilon\beta+1}(\frac{1}{A})\|f\|^2_{L^2(\mathbb{R}^d)}.
\end{eqnarray*}
This, via (\ref{equation0equiv}),  leads to our desired conclusion.\end{proof}

{\it Proof of Theorem \ref{thm1.1}: $L^p(\mathbb{R}^d)$ boundedness}. By duality, it suffices to prove that $T_{\Omega,a;\,k}$ is bounded on $L^p(\mathbb{R}^d)$ for $2<p<2\beta$.

For $j\in\mathbb{Z}$, let $K_j(x)=\frac{\Omega(x)}{|x|^{d+k}}\chi_{\{2^{j-1}\leq |x|<2^j\}}(x)$.
Let $\omega\in C^{\infty}_0(\mathbb{R}^d)$ be a nonnegative radial function such that
$${\rm supp}\, \omega\subset \{x:\,|x|\leq 1/4\},\,\,\,\int_{\mathbb{R}^d}\omega(x)dx=1,$$
and
$$\int_{\mathbb{R}^d}x^{\gamma}\omega(x)dx=0,\,\,\,1\leq |\gamma|\leq k.$$
For $j\in \mathbb{Z}$, set $\omega_j(x) = 2^{-dj}\omega(2^{-j}x)$. For a positive integer $l$, define
$$H_l(x)=\sum_{j\in\mathbb{Z}}K_j*\omega_{j-l}(x).$$
Let $R_l$ be the convolution operator with kernel $H_l$. For a function $a$ on $\mathbb{R}^d$ such that $\nabla a\in L^{\infty}(\mathbb{R}^d)$, recall that $[a,\,R_l]^k$ denotes the $k$-th commutator of $R_l$ with symbol $a$.

We claim that for each fixed $\varepsilon\in (0,\,1)$, $l\in\mathbb{N}$,
\begin{eqnarray}\label{equation0.12}
\|T_{\Omega, a;\,k}f-[a,\,R_{l}]^kf\|_{L^2(\mathbb{R}^d)}\lesssim l^{-\varepsilon\beta+1}\|f\|_{L^2(\mathbb{R}^d)}.
\end{eqnarray}
To prove this, write
$$H_l(x)-\sum_{j\in\mathbb{Z}}K_j(x)=\sum_{j\in\mathbb{Z}}\big(K_j(x)-K_j*\omega_{j-l}(x)\big)=:\sum_{j\in\mathbb{Z}}\mu_{j,l}(x).$$
By the vanishing moment of $\omega$, we know that for all multi-indices $\gamma\in\mathbb{Z}_+^d$ with  $1\leq |\gamma|\leq k$, $\partial^{\gamma}\widehat{\omega}(0)=0.$
By Taylor series expansion and the fact that $\widehat{\omega}(0)=1$, we deduce that
$$|\widehat{\omega}(2^{j-l}\xi)-1|\lesssim \min\{1,\,|2^{j-l}\xi|^{k+1}\}.
$$When $\Omega\in GS_{\beta}(S^{d-1})$ for some $\beta\in (1,\,\infty)$, it was proved in \cite[p. 458]{gras} that
$$
|\widehat{K_j}(\xi)|\lesssim 2^{-jk}\min\{1,\,\log ^{-\beta}(2+|2^j\xi|)\}.
$$
Thus, the Fourier transform estimate
\begin{eqnarray}\label{eq0.fourier}
&&|\widehat{\mu_{j,l}}(\xi)|=|\widehat{K_j}(\xi)||\widehat{\omega}(2^{j-l}\xi)-1|\lesssim 2^{-jk}\min\{\log^{-\beta}(2+|2^j\xi|),\,|2^{j-l}\xi|^{k+1}\}
\end{eqnarray}
holds true. On the other hand, a trivial computation shows that for all multi-indices $\gamma\in\mathbb{Z}_+^d$,
$$\|\partial^{\gamma}\widehat{K_j}\|_{L^{\infty}(\mathbb{R}^{d})}\lesssim \|\Omega\|_{L^1(S^{d-1})}2^{(|\gamma|-k)j},$$ and so for all $\xi\in\mathbb{R}^d$,
\begin{eqnarray}\label{equation0.derivative}
&&|\partial^{\gamma}\widehat{\mu_{j,l}}(\xi)|\lesssim \sum_{\gamma_1+\gamma_2=\gamma}|\partial^{\gamma_1}\widehat{K_j}(\xi)||\partial^{\gamma_2}\widehat{\omega}(2^{j-l}\xi)|\lesssim \|\Omega\|_{L^1(S^{d-1})}2^{j(|\gamma|-k)}.
\end{eqnarray}
The Fourier transforms (\ref{eq0.fourier})  and (\ref{equation0.derivative}), via Theorem \ref{thm0.21} with $A=2^{-l}$, lead to (\ref{equation0.12}) immediately.

Let $\varepsilon\in (0,\,1)$ be a constant which will be chosen later. An application of (\ref{equation0.12}) gives us that
\begin{eqnarray}\label{equation0.17}
\big\|[a,\,R_{2^l}]^kf-[a,\, R_{2^{l+1}}]^kf\big\|_{L^2(\mathbb{R}^d)}\lesssim 2^{(-\varepsilon\beta+1)l}\|f\|_{L^2(\mathbb{R}^d)}.
\end{eqnarray}
Therefore, the series
\begin{eqnarray}\label{equation0.15}
T_{\Omega,a;\,k}=[a,\,R_2]^k+\sum_{l=1}^{\infty}([a,\,R_{2^{l+1}}]^k-[a,\,R_{2^{l}}]^k)
\end{eqnarray}
converges in $L^2(\mathbb{R}^d)$ operator norm.

For $l\in\mathbb{N}$, let $L_l(x,\,y)=H_l(x-y)(a(x)-a(y))^k$. We claim that for any $y,\,y'\in\mathbb{R}^d$,
\begin{eqnarray}\label{equation0.hormander}&&
\int_{|x-y|\geq  2|y-y'|}|L_l(x,y)-L_l(x,y')|dx\\
&&\quad+\int_{|x-y|\geq  2|y-y'|}|L_l(y,x)-L_l(y',x)|dx\lesssim l.\nonumber\end{eqnarray}
To prove this, let $|y-y'|=r$. A trivial computation yields
\begin{eqnarray*}
\int_{|x-y|\geq 2r}\big|H_l(x-y)(a(y)-a(y'))^k\big|dx&\lesssim& r\sum_{j}\int_{|x|\geq 2r}|K_j*\omega_{j-l}(x)|dx\\
&\lesssim&r^k\sum_{j:\,2^{j-2}\geq r}\|K_j\|_{L^1(\mathbb{R}^d)}\|\omega_{j-l}\|_{L^1(\mathbb{R}^d)}\lesssim1,
\end{eqnarray*}
since $\|K_j\|_{L^1(\mathbb{R}^d)}\lesssim 2^{-j}$. For each fixed $j\in\mathbb{Z}$, observe  that
$$\|\omega_{j-l}(\cdot-y)-\omega_{j-l}(\cdot-y')\|_{L^1(\mathbb{R}^d)}\lesssim \min\{1,\,2^{l-j}|y-y'|\}.
$$
It then follows from  Young's inequality that
\begin{eqnarray*}
&&\int_{|x-y\geq 2r}|H_l(x-y)-H_l(x-y')||a(x)-a(y)|^kdx\\
&&\quad=\sum_{n=1}^{\infty}\int_{2^nr\leq |x-y\leq 2^{n+1}r}|H_l(x-y)-H_l(x-y')||a(x)-a(y)|^kdx\\
&&\quad\lesssim \sum_{n=1}^{\infty}(2^nr)^k\sum_{j:\,2^j\approx 2^nr}\|K_j\|_{L^1(\mathbb{R}^d)}
\|\omega_{j-l}(\cdot-y)-\omega_{j-l}(\cdot-y')\|_{L^1(\mathbb{R}^d)}\\
&&\quad\lesssim \sum_{k=1}^{\infty}\min\{1,\,2^{-k}2^{l}\}\lesssim l.
\end{eqnarray*}
Combining the estimates above gives us that
\begin{eqnarray*}
&&\int_{|x-y|\geq  2|y-y'|}|L_l(x,y)-L_l(x,y')|dx\\
&&\quad\leq \int_{|x-y|\geq 2r}\big|H_l(x-y)(a(y)-a(y'))^k\big|dx\\
&&\qquad+\int_{|x-y\geq 2r}|H_l(x-y)-H_l(x-y')||a(x)-a(y)|^kdx\lesssim l.
\end{eqnarray*}
Similarly, we can verify that
$$\int_{|x-y|\geq  2|y-y'|}|L_l(y,x)-L_l(y',x)|dx\lesssim l.$$
This establishes (\ref{equation0.hormander}).

Recall that  $T_{\Omega,\,a;k}$ is bounded on $L^2(\mathbb{R}^d)$. It follows from (\ref{equation0.12}) that $[a,\,R_l]^k$ is also bounded on $L^2(\mathbb{R}^d)$ with bound independent of $l$. This, along with (\ref{equation0.hormander}) and  Calder\'on-Zygmud theory, tells us that
$$
\big\|[a,\,R_l]^kf-[a,\, R_{l+1}]^kf\big\|_{L^p(\mathbb{R}^d)}\lesssim l\|f\|_{L^p(\mathbb{R}^d)},\,\,p\in (1,\,\infty),
$$ and so
\begin{eqnarray}\label{equation0.16}
\big\|[a,\,R_{2^l}]^kf-[a,\, R_{2^{l+1}}]^kf\big\|_{L^p(\mathbb{R}^d)}\lesssim 2^l\|f\|_{L^p(\mathbb{R}^d)},\,\,p\in (1,\,\infty).
\end{eqnarray}
Interpolating inequalities (\ref{equation0.17}) and (\ref{equation0.16}) shows that for any $\varrho\in (0,\,1)$ and $p\in (2,\,\infty)$,
\begin{eqnarray*}
\big\|[a,\,R_{2^l}]^kf-[a,\, R_{2^{l+1}}]^kf\big\|_{L^p(\mathbb{R}^d)}\lesssim 2^{(-2\varepsilon\beta/p+1+\varrho)l}\|f\|_{L^p(\mathbb{R}^d)}.
\end{eqnarray*}
For each $p$ with $2<p<2\beta$, we can choose $\varepsilon>0$ close to $1$ sufficiently, and $\varrho>0$ close to $0$ sufficiently, such that $2\varepsilon\beta/p-1-\varrho>0$. This, in turn, shows that
$$\sum_{l=1}^{\infty}\big\|[a,\,R_{2^l}]^kf-[a,\, R_{2^{l+1}}]^kf\big\|_{L^p(\mathbb{R}^d)}\lesssim \|f\|_{L^p(\mathbb{R}^d)},
$$
and the series (\ref{equation0.15}) converges in the
$L^p(\mathbb{R}^d)$ operator norm. Therefore, $T_{\Omega,\,a;k}$ is bounded on $L^p(\mathbb{R}^d)$ for $2<p<2\beta$. This finishes the proof of Theorem \ref{thm1.1}.\qed

\begin{remark}Let $\Omega$ be
homogeneous of degree zero, integrable and have mean value zero on
${S}^{d-1}$, $T_{\Omega}$ be the homogeneous singular integral operator defined by (\ref{equation1.4}). For $b\in{\rm BMO}(\mathbb{R}^d)$, define the commutator of $T_{\Omega}$ and $b$ by
$$[b,\,T_{\Omega}]f(x)=b(x)T_{\Omega}f(x)-T_{\Omega}(bf)(x).
$$
When $\Omega\in {\rm Lip}_{\alpha}(S^{d-1})$ with $\alpha\in (0,\,1]$, Uchiyama \cite{uch} proved that $[b,\,T_{\Omega}]$ is a compact operator on $L^p(\mathbb{R}^d)$ ($p\in (1,\,\infty)$) if and only if $b\in {\rm CMO}(\mathbb{R}^d)$, where ${\rm CMO}(\mathbb{R}^d)$ is the closure of $C^{\infty}_0(\mathbb{R}^d)$ in the ${\rm BMO}(\mathbb{R}^d)$ topology,
which coincide with the space of functions of vanishing mean oscillation. When $\Omega\in GS_{\beta}(S^{d-1})$ for $\beta\in (2,\,\infty)$,  Chen and Hu \cite{chenhu} considered the compactness of $[b,\,T_{\Omega}]$ on $L^p(\mathbb{R}^d)$ with $\beta/(\beta-1)<p<\beta$. For other work about the compactness of $[b,\,T_{\Omega}]$, see \cite{taoyang} and the references therein.  It is of interest to characterize the compactness of Calder\'on commutator $T_{\Omega,\,a;\,k}$ on $L^p(\mathbb{R}^d)$ ($p\in (1,\,\infty)$). We will consider this in a forthcoming paper.
\end{remark}

\medskip

{\bf Acknowledgement}. The authors would like to express their sincerely thanks to
the referee for his/her valuable remarks and suggestions, which made this paper
more readable. Also, the authors would like to thank professor Dashan Fan for helpful suggestions and comments.

\end{document}